\documentclass[12pt,a4paper]{article}

\usepackage[dutch,english]{babel}
\usepackage{amssymb, amsmath, color, fancyhdr, enumitem, amsthm, hyperref, graphicx, picins, array, tikz-cd, tikz}
\usetikzlibrary{arrows}
\usepackage[toc,page]{appendix}
\usepackage[a4paper, top=3.5cm, bottom=3.5cm, left=3.5cm, right=3.5cm]{geometry}
\usepackage[all]{xy}
\newtheorem{theorem}{Theorem}
\newtheorem*{theorem*}{Theorem}
\newtheorem{lemma}[theorem]{Lemma}
\newtheorem{proposition}[theorem]{Proposition}

\newtheorem*{corollary*}{Corollary}
\theoremstyle{definition}
\newtheorem{definition}[theorem]{Definition}
\newtheorem{remark}[theorem]{Remark}
\newtheorem{example}[theorem]{Example}
\newtheorem*{example*}{Example}
\newtheorem{setup}[theorem]{Set-up}
\renewcommand{\O}{\mathcal{O}}

\newcommand{\comment}[1]{}

\newcommand{\Z}{\ensuremath{\mathbb{Z}}}

\newcommand{\F}{\ensuremath{\mathbb{F}}}

\renewcommand{\P}{\ensuremath{\mathbb{P}}}

\renewcommand{\O}{\mathcal{O}}

\newcommand{\Spec}{\mathop\mathrm{Spec}}

\newcommand{\stacks}[1]{\textrm{\cite[\href{http://stacks.math.columbia.edu/tag/#1}{#1}]{stacks}}}

\let\oldproof\proof
\def\proof{\oldproof\unskip}

\newcounter{nootje}
\renewcommand\check[1]{}

\begin{document}
\title{Inverse Galois problem for ordinary curves}
\author{Raymond van Bommel}
\date{}
\maketitle

{\bf Abstract.} We consider the inverse Galois problem over function fields of positive characteristic $p$, for example, the inverse Galois problem over the projective line. We describe a method to construct certain Galois covers of the projective line and other curves, which are ordinary in the sense that their Jacobian has maximal $p$-torsion. We do this by constructing Galois covers of ordinary semi-stable curves, and then deforming them into smooth Galois covers.


\section{Introduction}

In \cite{GlassPries}, Glass and Pries prove that there is an ordinary (in the sense that their Jacobians have maximal $p$-torsion) hyperelliptic curve of every genus in characteristic $p > 2$. Viewing hyperelliptic curves as $\Z/2\Z$-covers of the projective line $\P^1$, this leads to the following question: is it possible, for any finite group $G$, to construct an ordinary curve which is a Galois cover of $\P^1$ whose Galois group is $G$?

The inverse Galois problem over function fields has been studied extensively. In \cite{Harbater84} and \cite{Harbater87}, Harbater solved the problem for $\overline{\F}_p(t)$ and for $k(t)$ where $k$ is a complete ultrametric field (see also \cite[Sect.\ V.2]{IGT}). The problem is solved by explicitly constructing covers of $\P^1$ over $\overline{\F}_p$ or $k$ using rigid analytic methods. These methods, however, do not seem to give us a way to easily determine whether the curves constructed are ordinary. 

In this article, we will construct Galois covers which are ordinary using a different method: deformation theory. We will reprove the aforementioned result by Glass and Pries in a different way, that will also allow us to construct ordinary covers of $\P^1$ (and other curves) with other Galois groups.

For our method, we will consider Galois covers of ordinary semi-stable curves. We will then use deformation theory on semi-stable curves, \cite{DeligneMumford}, and in particular on curves with an action of a group, \cite{Saidi}, to smoothen the cover. Our key result in this direction is the following.

\begin{theorem*}\label{curvedeform} %[Thm.\ \ref{curvedeform}, p.\ \pageref{curvedeform}]
Let $G$ be a finite group and let $\gamma : C \rightarrow D$ be a Galois cover (cf.\ Def.\ \ref{def:Galoiscover}, p.\ \pageref{def:Galoiscover}), with Galois group $G$, of semi-stable curves over an algebraically closed field $k$ of characteristic $p$ coprime to the order of $G$. Moreover, suppose that for each $P \in C^{\mathrm{sing}}$ and each $\sigma \in \mathrm{Stab}(P) \subset G$, the determinant of the action of $\sigma$ on the tangent space at $P$ is $-1$ if $\sigma$ swaps the two branches of $C$ at $P$ and 1 otherwise. Then there exists a smoothening $\Gamma : \mathcal{C} \rightarrow \mathcal{D}$ of the cover $\gamma$, where $\mathcal{C}$ and $\mathcal{D}$ are semi-stable curves over $k[[X]]$, such that $\mathcal{C}$ is ordinary. In partiucular if $C$ is ordinary, then $\mathcal{C}$ is ordinary, respectively.
\end{theorem*}

%For example, using a configuration of copies of $\P^1$, whose associated graph is the Petersen graph, we will realise the alternating group $A_5$.
%\begin{corollary*}[Thm.\ \ref{example:A5}, p.\ \pageref{example:A5}]
%Let $p > 5$ be a prime number. Then there exists a field $K$ of characteristic $p$ and a Galois cover $C \rightarrow \P^1$ with group $A_5$ such that $C$ is a smooth ordinary curve.
%\end{corollary*}

For example, using a configuration of copies of $\P^1$, whose associated graph is an $n$-gon, we will realise the dihedral group $D_n$.
\begin{example*}[Ex.\ \ref{example:ngon}, p.\ \pageref{example:ngon}]
Let $n$ be a positive integer and let $p$ be a prime number not dividing $2n$. Then there exists a Galois cover $C \rightarrow \P^1$ with group $D_n$ such that $C$ is a smooth ordinary curve over $\overline{\F}_p$.
\end{example*}

%For example, we will realise the dihedral group $D_n$ in the following way.
%by taking a semi-stable curve consisting of $n$ projective lines, glued in the shape of an $n$-gon with the natural action of $D_n$, and then deforming it.
%\begin{corollary*}[Thm.\ \ref{example:ngon}, p.\ \pageref{example:ngon}]
%Let $n$ be an integer and let $p$ be a prime number not dividing $2n$. Let $C$ be the semi-stable curve over $k = \overline{\F}_p$ obtained by gluing $n$ copies of the projective line in the shape of an $n$-gon. Consider the natural action of $D_n$ on $C$. Then the cover $C \rightarrow C/D_n \cong \P^1_{k}$ can be deformed into a cover over $k[[X]]$ as above, proving that there is a Galois cover $\mathcal{C}_{\eta} \rightarrow\P^1_{k((X))}$ with Galois group $D_n$, such that $\mathcal{C}_{\eta}$ is smooth and ordinary.
%\end{corollary*}

In section 2, we will treat the definitions necessary to extend the notion of ordinarity to semi-stable curves. In section 3, we will look at graphs associated to semi-stable curves and their quotients by group actions. In section 4, we show how to deform these singular curves to smooth curves, while preserving the group action. The proof of the main theorem will be given in section 5. Finally, in section 6, we will list some examples of groups for which we can construct ordinary Galois covers, using this method.

{\bf Acknowledgements.} The author wishes to thank his PhD advisors David Holmes and Fabien Pazuki, and also Bas Edixhoven, Maarten Derickx and Peter Koymans, for their thorough reading of this paper and their useful comments for improvements.

%% SECTION : ORDINARITY OF SEMI-STABLE CURVES %%
\section{Ordinarity of semi-stable curves}

Let us first recall the notion of semi-stable curves.

\begin{definition}
For an algebraically closed field $k$, a {\em semi-stable curve over $k$} is a scheme $C/k$ that is reduced, connected, one-dimensional, projective over $\Spec{k}$ (i.e.\ {\em an algebraic curve}) whose singular points are nodal, i.e.\ they are ordinary double points. More generally, for a scheme $S$, a scheme $C/S$ is said to be a {\em semi-stable curve over $S$} if $C_{s}$ is a semi-stable curve over $s$ for every geometric point $s$ of $S$.
\end{definition}

Let $k$ be an algebraically closed field of characteristic $p > 0$ and let $C / \Spec{k}$ be a semi-stable curve. Consider the map $\O_C \rightarrow \O_C$ raising all sections to the power $p$ and let $F : H^1(C, \O_C) \rightarrow H^1(C, \O_C)$ be the induced map on cohomology. This map satisfies the conditions of \cite[Cor., chap.\ 14, p.\ 143]{MumfordAV}, hence $H^1(C,\O_C)$ decomposes as the sum of a semisimple part and a nilpotent part. Now we use this to define the notion of ordinarity for semi-stable curves.

\begin{definition}[\textrm{\cite[sect.\ 15, p.\ 146--150]{MumfordAV}}]\label{defordinary}
Let $k$ be an algebraically closed field of characteristic $p > 0$. Let $C$ be a semi-stable curve over $k$ and, moreover, let $F : H^1(C,\O_C) \rightarrow H^1(C,\O_C)$ be the map induced by the absolute Frobenius.  Then $C$ is called {\em ordinary} if the semisimple rank of $F$ is maximal or, in other words, if the dimension of the semisimple part of $H^1(C,\O_C)$ is equal to $\dim_k(H^1(C,\O_C))$. More generally, for a scheme $S$, a semi-stable curve $C/S$ is called {\em ordinary} if $C_s$ is ordinary for every geometric point $s$ of $S$.
\end{definition}

\begin{remark}It is well-known that for smooth curves this definition coincides with the classical definition, i.e.\ a smooth curve is ordinary of the $p$-torsion of the Jacobian of the curve is maximal.\end{remark}

An easy well-known corollary of \cite[Cor., chap.\ 14, p.\ 143]{MumfordAV} is the following lemma.

\begin{lemma}\label{ordisom}
Let $C / \Spec{k}$ and $F$ be as before. Then $C$ is ordinary if and only if $F$ is an isomorphism. 
\end{lemma}

The goal of this section is to prove that a semi-stable curve is ordinary (in the sense of Definition \ref{defordinary}) if and only if all irreducible components of its normalisation are.

\begin{proposition}\label{ordinarycomponents}
Let $C / k$ be a semi-stable curve over an algebraically closed field. Let $C_1, \ldots, C_n$ be the irreducible components of its normalisation. Then $C$ is ordinary if and only if for all $i = 1, \ldots, n$ the smooth curve $C_i$ is ordinary.
\end{proposition}

\begin{proof}
Consider the exact sequence
$$0 \longrightarrow \O_C \longrightarrow \pi_* \O_{\widetilde{C}} \longrightarrow \pi_* \O_{\widetilde{C}} / \O_C \longrightarrow 0.$$
The sheaf $\pi_* \O_{\widetilde{C}} / \O_C$ is a direct sum of skyscraper sheaves, one for each singular point $P$ of $C$, having $k$ as stalk in $P$. As $\pi$ is an affine morphism, we have $H^i(C, \pi_* \mathcal{F}) = H^i(\widetilde{C}, \mathcal{F})$ for any quasi-coherent sheaf $\mathcal{F}$ of $\O$-modules on $\widetilde{C}$. Hence, for the associated long exact sequence we get
$$0 \longrightarrow k \longrightarrow k^n \longrightarrow k^{\mathrm{Sing}(C)} \longrightarrow H^1(C,\O_C) \longrightarrow \bigoplus_{i=1} ^n H^1(C_i, \O_{C_i}) \longrightarrow 0.$$
Now consider the action of Frobenius on this exact sequence. On $k, k^n$ and $k^{\mathrm{Sing}(C)}$ it acts componentwise, hence it induces an isomorphism. By Lemma \ref{ordisom}, on the one hand, Frobenius acts on $H^1(C,\O_C)$ as an isomorphism if and only if $C$ is ordinary. On the other hand, Frobenius acts on $H^1(C_i, \O_{C_i})$ as an isomorphism if and only if $C_i$ is ordinary. Now we can use the five-lemma to conclude the desired result.
\end{proof}

%% SECTION : QUOTIENTS OF CURVE %%
\section{Quotients of curves}

In this section, we will build a formalism for graphs associated to semi-stable curves. Given a semi-stable curve with a group $G$ acting on it, we want to build an action of $G$ on the associated graph, and define a good quotient in the category of graphs. The key property we need is that the graph of the quotient curve is naturally the quotient of the graph.
The following two examples will serve as motivation for the definition that follows.

Throughout this section $S$ will be $\Spec{k}$, the spectrum of an algebraically closed field.

\begin{example}\label{ex1:C3}
Let $C/S$ be a curve obtained by gluing three copies of $\P^1$ to form a triangle. We let the group $G = \Z/3\Z$ act on $C$ by cyclically permuting these curves. Then the quotient of $C$ by $G$ is one copy of $\P^1$ with two of its points glued together.
\end{example}

\begin{example}\label{ex2:S3}
Let $C/S$ be the same curve as in the previous example. We let the group $G = S_3$ act on $C$, by cyclically permuting and mirroring the whole curve in its three symmetry axes. Then the quotient of $C$ by $G$ is isomorphic to $\P^1$.
\end{example}

\begin{definition}[Graph]
A {\em graph} consists of the data of a set of vertices $V$, for every $v \in V$ a set $E_v$ of edge ends, an element\footnote{The elements $\emptyset_v$ of the $E_v$ may seem to be a bit out of place. The reason they exists is because morphisms of semi-stable curves can map singular points to smooth points. This will be clarified further in Remark \ref{emptysetpt}.} $\emptyset_v \in E_v$, and an involution $n : \bigsqcup_v E_v \rightarrow \bigsqcup_v E_v$, giving the opposite of each edge end, whose only fixed points are the $\emptyset_v$.

A morphism of graphs from $(V, (E_v)_v, n_V)$ to $(W, (E_w)_w, n_W)$ consists of a morphism $\varphi : V \rightarrow W$ and morphisms $\psi_v : E_v \rightarrow E_{\varphi(v)}$ for every $v \in V$, gluing to a morphism $\psi : \bigsqcup_v E_v \rightarrow \bigsqcup_w E_w$ such that $\psi \circ n_V =  n_W \circ \psi$.
\end{definition}

\begin{remark}
Colimits exist in the category of graphs. For a diagram of graphs the colimit is constructed as follows. The set of vertices $V$ is the colimit of the sets of vertices in the diagram. For each vertex $v \in V$, we consider the collection of vertices $w$ in the original diagram that map to $v$. Its sets of edge ends form a diagram $(E_w)_{w \mapsto v}$. Then we take the colimit $E_v'$ of this diagram, together with the map $n_V': E_v' \rightarrow E_v'$ obtained by gluing the $n_W$. Now $E_v$, the set of edge ends for $v$, is obtained from $E_v'$ by identifying all edge ends which are fixed points of $n_V'$. The element $\emptyset_v$ is the image of $\emptyset_w$ for any $w$ mapping to $v$. The morphism $n_V$ is directly obtained from $n_V'$. One can easily check that this indeed defines a graph and that this is the colimit of the diagram.

In particular, for a graph $\Gamma$ equipped with an action of a group $G$, the quotient graph $\Gamma/G$ can be shown to exist by taking the colimit of the diagram whose unique object is $\Gamma$ and whose arrows are given by the elements of $G$.
\end{remark}
 
\begin{definition}[Associated graph of a semi-stable curve]
Let $C$ be a semi-stable curve over $S$. Then we can associate to it a graph, whose vertices are the irreducible components. Besides the $\emptyset_v$ for each vertex $v$, there are also two edge ends for each singular point $P$: one for each of the two (possibly the same) components  that are intersecting there. The involution $n$ swaps the two edge ends corresponding to $P$.
\end{definition}

\begin{example}
The graph $\Gamma(C)$ corresponding to the curve $C$ from Examples \ref{ex1:C3} and \ref{ex2:S3} consists of three vertices $C_1, C_2$ and $C_3$. The set of edge ends for $C_i$ is $\{C_{i,1}, C_{i,2}, C_{i,3}\}$, where $\emptyset_i = C_{i,i}$ and $n(C_{i,j}) = C_{j,i}$. In both cases, the graph associated to the quotient curve (a single point with a loop, and a single point, respectively) is the quotient of $\Gamma(C)$ by the action of $\Z/3\Z$ and $S_3$, respectively.
\end{example}

\begin{proposition} \label{sschar}
Let $C/S$ be a possibly singular reduced curve. Let $\pi: \widetilde{C} \rightarrow C$ be its normalisation. Then $C$ is semi-stable if and only if there exist two maps $\ell_1, \ell_2: L \rightarrow \widetilde{C}$ from the singular locus of $C$ to $\widetilde{C}$ such that $\pi \circ \ell_1 = \pi \circ \ell_2$ is the inclusion of the singular locus in $C$ and $C$ is the colimit of the following diagram with these two arrows: $$L \rightrightarrows \widetilde{C}.$$
\end{proposition}
\begin{proof}
If $C$ is semi-stable, then there lie two points above each singular point of $C$ and by ordering them for each singular point, we get two maps $\ell_1$ and $\ell_2$ as required. Then by \cite[Prop.\ 7.5.15, p.\ 310]{Liu}, $C$ is the colimit of the diagram.

If $C$ is such a colimit, then above each singular point of $C$ there are one or two points. If there is only one, then $C$ would be smooth at that point, which is not the case, hence there are two. Then by \cite[Prop.\ 7.5.15, p.\ 310]{Liu} the singularities are nodal and we are done.
\end{proof}

\begin{proposition}\label{quotientofcurves}
Let $C/S$ be a semi-stable curve. Let $G \subset \mathrm{Aut}_S(C)$ be a finite group. Let $\Gamma$ be the graph associated to $C$. Then there exists a semi-stable curve $D/S$ that is a categorical quotient of $C$ by $G$ in the category of schemes. Its associated graph is $\Gamma/G$.
\end{proposition}

\begin{proof}
Let $\pi : \widetilde{C} \rightarrow C$ be the normalisation of $C$. For each singular point $p \in C$ choose an ordering on the two points in $\pi^{-1}(p)$. Let $L$ be the union of the singular points of $C$. Then the ordering chosen gives us two maps $\ell_1, \ell_2: L \rightarrow \widetilde{C}$ mapping $p$ to the first respectively second point in $\pi^{-1}(c)$. Observe, cf.\ Proposition \ref{sschar}, that $C$ is the colimit of the diagram 
$$L \rightrightarrows \widetilde{C},$$
where the two maps are $\ell_1$ and $\ell_2$.

The universal property of the normalisation gives us an extension of the action of $G$ on $C$ to $\widetilde{C}$. Now the irreducible components of $\widetilde{C}$ are connected and $G$ permutes them. Using \cite[Thm.\ 1, chap.\ 12, p.\ 111]{MumfordAV}, we see that the categorical quotient $q : \widetilde{C} \rightarrow \widetilde{D}$ of $\widetilde{C}$ by the group action of $G$ identifies components sent to each other and each such component is replaced by the quotient of that component by the action of its stabiliser. Now we will recall the standard argument to prove that these quotients of the components are smooth curves.

As a subring of an integral ring is integral, the quotients of the components are integral. Furthermore, they are normal, as for an integrally closed domain $A$ with an action of a group $G$, the ring $A^G$ is integrally closed. As $A$ is an integral extension of $A^G$, it has the going-up property and $A^G$ is noetherian and of dimension 1. Hence, the quotients of the components are normal integral of dimension 1. Hence, the quotients of the components of $\widetilde{C}$ are smooth curves, see also \stacks{0BX2}.

Now we will re-glue the points. Let $D$ be the colimit of the diagram 
$$L \rightrightarrows \widetilde{D},$$
where the two maps are $q \circ \ell_1$ and $q \circ \ell_2$. Let $\zeta: \widetilde{D} \rightarrow D$ be the natural map. Now, we will prove that this is also a categorical quotient of $C$ by $G$.

$$\xymatrix{
L \ar@/^/[r]^{\ell_1} \ar@/_/[r]_{\ell_2}& \widetilde{C} \ar[d]^{\pi} \ar[r]^q &\widetilde{D} \ar[d]^{\zeta}\ar@/^5pc/@{-->}[ddl]^\rho \\
&C \ar@{-->}[r]^Q \ar[d]^\psi &D \ar@{-->}[dl]^{\nu} \\ &T
}$$

First of all, the maps $\zeta \circ q \circ \ell_1$ and $\zeta \circ q \circ \ell_2$ are the same by construction. Hence, there is a map $Q: C \rightarrow D$, such that $Q \circ \pi = \zeta \circ q$.

Let $T$ be a test scheme with a trivial action of $G$ and $\psi : C \rightarrow T$ be a $G$-equivariant map.  Then the map $\psi \circ \pi: \widetilde{C} \rightarrow T$ is also $G$-equivariant and gives rise to a unique map $\rho : \widetilde{D} \rightarrow T$, such that $\rho \circ q = \psi \circ \pi$. Now we have $$\rho \circ q \circ \ell_1 = \psi \circ \pi \circ \ell_1 = \psi \circ \pi \circ \ell_2 =\rho \circ q \circ \ell_2,$$ and hence by the universal property of the colimit there is a unique map $\nu: D \rightarrow T$ such that $\nu \circ \zeta = \rho$. Now the compositions $\psi \circ \pi$ and $\nu \circ \zeta \circ q = \nu \circ Q \circ \pi$ are equal. Since the normalisation is surjective on schemes and injective on the underlying rings, it is an epimorphism. This yields $\psi = \nu \circ Q$ as desired.

Now the statement follows by using Proposition \ref{sschar}.
\end{proof}

\begin{remark}\label{emptysetpt}
These methods can also be generalised to the case of marked semi-stable curves. The distinguished edge ends $\emptyset_v$ should be removed, and instead marked points should correspond to fixed points of the involution $n$ on the edge ends.
Furthermore, it is necessary to require that the marked points are smooth and that morphisms map singular points to singular or marked points.
\end{remark}

%% SECTION : DEFORMATIONS OF G-COVERS %%
\section{Deformations of $G$-covers}

In this section, we will treat deformations of Galois covers. Let us first define this notion.

\begin{definition}[Galois cover] \label{def:Galoiscover}
Let $f: X \rightarrow Y$ be a morphism of noetherian schemes or noetherian formal schemes, and let $G$ be (isomorphic to) a finite group of automorphisms of $X$. Then $f$ is called a {\em Galois cover with group $G$} if:\\[-0.9cm]
\begin{itemize}\itemsep0pt
\item $Y$ as a topological space is the quotient of $X$ by $G$,
\item for each open $U \subset Y$ we have $\O_Y(U) \cong \O_X(f^{-1}(U))^G$, and
\item each $g \in G \setminus \{\mathrm{id}\}$ does not act as the identity on any of the irreducible components of $X$.
\end{itemize}
\end{definition}

\begin{setup} \label{setupquotcurve}
Let $C$ be a semi-stable curve over $S = \Spec{k}$, the spectrum of an algebraically closed field. Let $G \subset \mathrm{Aut}_S(C)$ be a finite group of order coprime to $p := \mathop\mathrm{char}{k}$ such that for each component of $C$ the stabiliser acts faithfully on the component. Let $D$ be the quotient of $C$ by $G$ (as in Proposition \ref{quotientofcurves}). Then $\gamma: C \rightarrow D$ is a Galois cover with group $G$.
\end{setup}

\begin{remark} \label{localdescriptioncover}
Suppose we are in the setting of Set-up \ref{setupquotcurve}. Let $Q \in C^{\mathrm{sing}}$ be a node and let $P \in D$ be its image $\gamma(Q)$. Then $P$ could be either singular or smooth. We will describe the local structure of the covering $\gamma$ in both cases.

\begin{itemize}
\item If $P \in D^{\mathrm{sing}}$, then we know that $\widehat{\O}_{C,Q} \cong k[[a,b]]/(ab)$ and moreover $\widehat{\O}_{D,P} \cong k[[s,t]]/(st)$ (cf.\ \cite[Prop. 5.15, p.\ 310]{Liu}) and, swapping the variables if necessary, one easily checks that the original cover is given by
$$\widehat{\gamma}_Q : k[[s,t]]/(st) \rightarrow k[[a,b]]/(ab) : \quad s \mapsto a^m \cdot u, \quad t \mapsto b^\ell \cdot v,$$ for some integers $m$ and $\ell$ and units $u \in k[[a]]$ and $v \in k[[b]]$. The stabiliser $\mathrm{Stab}(Q) \subset G$ acts faithfully on both components through $Q$. Hence, we find that $k((a))/k((s))$ and $k((b))/k((t))$ are Galois extensions of degree $m = |\mathrm{Stab}(Q)| = \ell$. In particular, $m$ is not divisible by $p$ and, by Hensel's lemma, $u$ and $v$ are $m$-th powers in $k[[a,b]]/(ab)$ and we may and will assume, by composing with an automorphism if necessary, that $u=v = 1$.

In particular, the group $\mathrm{Stab}(Q)$ is cyclic and acts on $a$ and $b$ by multiplication with $m$-th roots of unity.

\item
In the case that $P$ is a smooth point of $D$, the local cover $\widehat{\gamma}_Q$ is given by
$$\widehat{\gamma}_Q : k[[s]] \rightarrow k[[a,b]]/(ab) : \quad s \mapsto a^m \cdot u + b^{\ell} \cdot v,$$ for some integers $m$ and $\ell$ and units $u \in k[[a]]$ and $v \in k[[b]]$. Analogous to the first case, we get that $m = \ell = \tfrac12 \cdot |\mathrm{Stab}(Q)|$ and we may assume that $u, v = 1$.

The group $\mathrm{Stab}(Q)$ is either cyclic, generated by an element $g$ swapping the two components, such that $g^2$ acts by multiplication with primitive $m$-th roots of unity, or $\mathrm{Stab}(Q)$ is isomorphic to $\Z/2\Z \times \Z/m\Z$, where the first factor is acting by swapping the two components and the second one by multiplication with $m$-th roots of unity.
\end{itemize}
\end{remark}

In the part that follows, we will assume that the action of $G$ is orientation preserving in the following sense.

\newcommand{\Stab}{\mathrm{Stab}}
\begin{definition}\label{deforientationpreserving}
Let $k$ be an algebraically closed field and let $C$ be a semi-stable curve over $\Spec{k}$. Let $G \subset \mathrm{Aut}_k(C)$ be a finite group. For every $P \in C^{\mathrm{sing}}$, the subgroup $\Stab(P)$ acts on the completed local ring $(\widehat{\O}_{C,P}, \mathfrak{m}_P)$ of $C$ at $p$, and on its cotangent space $\mathfrak{m}_P / \mathfrak{m}_P^2$. The action of $G$ on $C$ is called {\em orientation preserving at $P$}, if for each $\sigma \in \Stab(P)$, this action of $\sigma$ on $\mathfrak{m}_P / \mathfrak{m}_P^2$ has determinant $\mathrm{sgn}_P(\sigma)$, where $\mathrm{sgn}_P(\sigma)$ is the sign of the action of $\sigma$ on the set consisting of the two edge ends of $\Gamma(C)$ corresponding to $P$. The action of $G$ on $C$ is called {\em orientation preserving} if it is orientation preserving at every $P \in C^{\mathrm{sing}}$.
\end{definition}

\begin{remark}
As we saw in Remark \ref{localdescriptioncover}, this condition is easily seen to be verified in case $\mathrm{Stab}(P)$ is isomorphic to either $1, \Z/2\Z$ or $(\Z/2\Z) \times (\Z/2\Z)$ for each $P \in C^{\mathrm{sing}}$. In these cases, the action of $\mathrm{Stab}(P)$ on $\widehat{\O}_{C,P} \cong k[[a,b]]/(ab)$ only involves swapping $a$ and $b$ and multiplying them simultaneously by $-1$.
\end{remark}

The following proposition will characterise orientation-preserving actions.

\begin{proposition}\label{orientationpreservinglifting}
Suppose we are in the setting of Set-up \ref{setupquotcurve}, and let $P \in C^{\textrm{sing}}$. Then the action of $G$ on $C$ is orientation preserving at $P$ if and only if there exists an isomorphism $\widehat{\O}_{C,P} \cong k[[a,b]]/(ab)$, such that the induced action of $\Stab(P)$ can be lifted to $k[[a,b]]$, in such a way that it stabilises $ab$.
\end{proposition}

\begin{proof}
If the action of $\Stab(P)$ can be lifted in the way described, then the action of each $\sigma \in \Stab(P)$ on the cotangent space has determinant equal to 1 if $\sigma$ maps $a$ to a multiple of $a$ and $b$ to a multiple of $b$, and determinant $-1$ otherwise. Now, let us prove the converse.

Suppose that the action of $G$ on $P$ is orientation preserving at $P$. By Remark \ref{localdescriptioncover}, we can find an isomorphism $\widehat{\O}_{C,P} \cong k[[a,b]]/(ab)$, such that each $\sigma \in \Stab(P)$ acts on $k[[a,b]]/(ab)$ by $a \mapsto \mu_{\sigma,a} \cdot a$ and $b \mapsto \mu_{\sigma,b} \cdot b$ in case $\mathrm{sgn}_P(\sigma) = 1$ (or vice versa in case $\mathrm{sgn}_P(\sigma) = -1$), for certain roots of unity $\mu_{\sigma,a}, \mu_{\sigma,b} \in k$. Then, the determinant of the action of $\sigma$ on the cotangent space is $\mathrm{sgn}_P(\sigma) \cdot \mu_{\sigma, a} \cdot \mu_{\sigma, b}$, which is $\mathrm{sgn}_P(\sigma)$ by assumption. Hence, in both cases we have $\mu_{\sigma,a} \cdot \mu_{\sigma,b} = 1$.

Then we can lift the action of $\sigma$ on $k[[a,b]]/(ab)$ to $k[[a,b]]$ by making the choice $a \mapsto \mu_{\sigma, a} \cdot a$ and $b \mapsto \mu_{\sigma, b} \cdot b$ in case $\mathrm{sgn}_P(\sigma) = 1$ (or vice versa in case $\mathrm{sgn}_P(\sigma) = -1$). Under this action, the element $ab$ is mapped to $\mu_{\sigma, a} \cdot \mu_{\sigma, b} \cdot ab = ab$, as desired.
\end{proof}

Now we get to the main result of this section.  We will prove that we can deform $\gamma$ to a cover of smooth curves in the following sense. 

\begin{proposition}\label{coverdeformation}
Suppose we are in the situation of Set-up \ref{setupquotcurve} and that the action of $G$ on $C$ is orientation preserving. Over the complete discrete valuation ring $R = k[[X]]$, with residue field $k$, there exists a $G$-Galois cover $\Gamma: \mathcal{C} \rightarrow \mathcal{D}$ of semi-stable curves over $\Spec{R}$ such that $\Gamma_s \cong \gamma$, and $\mathcal{C}_{\eta}$ and $\mathcal{D}_{\eta}$ are smooth, where $\eta$ and $s$ are the generic and special point of $\Spec{R}$, respectively.
\end{proposition}

\begin{proof}
First we shall consider the local structure of the map $\gamma$ above the singular points of $D$ in order to define the scheme $\mathcal{D}$ using deformation theory of stable curves. The curve $D$ is semi-stable and, as $k$ is algebraically closed, we can consider $D$ as a stable curve by marking some extra points on the components with low genus if necessary.

In Remark \ref{localdescriptioncover}, we already saw a description of the completed local rings above the singular points of $D$. Now we take $\mathcal{D} / \Spec{R}$ to be any deformation of $D$ such that for each singular point $P \in D^{\mathrm{sing}}$ as above, we have $\widehat{\O}_{\mathcal{D},P} \cong k[[s,t,X]]/(st - X^m)$, using \cite[Prop.\ 4.37, p.\ 117]{Bertin13}, where $m$ is as in Remark \ref{localdescriptioncover}.

To lift $\gamma$ we will use \cite[Prop.\ 1.2.4, p.\ 8]{Saidi}. Let $Z := D^{\mathrm{sing}} \cup D^{\mathrm{ram}}$, where $D^{\mathrm{ram}}$ is the locus where $\gamma$ is ramified. For each $P \in Z$, we will describe a Galois cover $\mathcal{C}_P \rightarrow \mathop\mathrm{Spf} \widehat{\O}_{\mathcal{D},P}$ lifting the cover $\widehat{C}_P \rightarrow \Spec{\widehat{\O}_{D,P}}$, where $\widehat{C}_P$ is the completion of $C$ above $P$. We will consider three cases.

\begin{itemize}
\item[(i)]
If both $P$ and the points $Q \in \gamma^{-1}(P)$ are smooth, then for each $Q$, the local cover $\widehat{\gamma}_Q$ is, after composition with an isomorphism if necessary, of the form $$k[[s]] \rightarrow k[[a]] : \quad s \mapsto a^m,$$ for some $m$ not divisible by $p$. The group $\mathrm{Stab}(Q)$ is acting as multiplication with powers of $m$-th roots of unity. Then we lift this part of the cover to $k[[X]]$ by taking the map
$$k[[s,X]] \rightarrow k[[a,X]] : \quad s \mapsto a^m, \quad X \mapsto X.$$
This is again a Galois cover (for the group $\mathrm{Stab}(Q)$) and for $\widehat{C}_P$ we take a disjoint union of these, one for each $Q \in \gamma^{-1}(P)$, or equivalently one for each element $g \in G/\mathrm{Stab}(Q)$, letting the Galois group act in the natural way.
\item[(ii)]
If both $P$ and the points of $Q \in \gamma^{-1}(P)$ are singular, then we saw in Remark \ref{localdescriptioncover} that for each $Q$, $\widehat{\gamma}_Q$ is of the form $$k[[s,t]]/(st) \rightarrow k[[a,b]]/(ab): \quad s \mapsto a^m, \quad t \mapsto b^m,$$ for some $m$ not divisible by $p$. The group $\mathrm{Stab}(Q)$ is acting by multiplication with powers of $m$-th roots of unity. We lift this part of the cover to $k[[X]]$ by taking the map
\begin{align*}k[[s,t,X]] / (st - X^m) &\rightarrow k[[a,b,X]] / (ab - X)\\
s \mapsto a^m, \quad &t \mapsto b^m, \quad X \mapsto X,\end{align*}
observing that we use Proposition \ref{orientationpreservinglifting} in order to make this lift. We then continue as in case (i), taking multiple copies of this cover for $\widehat{C}_P$.
\item[(iii)]
Finally, if $P$ is smooth but the points $Q \in \gamma^{-1}(P)$ are singular, then we saw in Remark \ref{localdescriptioncover} that for each $Q$, $\widehat{\gamma}_Q$ is, after composition with an isomorphism if necessary, of the form $$k[[s]] \rightarrow k[[a,b]]/(ab) : \quad s \mapsto a^m + b^m,$$ for some $m$ not divisible by $p$. The group $\mathrm{Stab}(Q)$ is acting by swapping $a$ and $b$ and multiplication with powers of $m$-th roots of unity. We lift this part of the cover to $k[[X]]$ by taking the map
$$k[[s,X]] \rightarrow k[[a,b,X]] / (ab - X) : s \mapsto a^m+b^m, \quad X \mapsto X,$$
again using Proposition \ref{orientationpreservinglifting}. We then continue as in case (i), taking multiple copies of this cover for $\widehat{C}_P$.
\end{itemize}

Using \cite[Prop.\ 1.2.4, p.\ 8]{Saidi}, we find a $G$-Galois cover $\mathcal{C} \rightarrow \mathcal{D}$ of schemes over $\Spec{(k[[X]])}$, such that the special fibre is $\gamma$ and locally above points $P$ as above the cover is described by maps
$$\widehat{\pi}_P : \widehat{C}_P \rightarrow \Spec{\widehat{\O}_{D,P}}$$
that we defined in the previous paragraphs. Smoothness of $\mathcal{C}_{\eta}$ and $\mathcal{D}_{\eta}$ over $k((X))$ follows from the following lemma. 
\end{proof}

\begin{lemma}
Let $\mathcal{C} \rightarrow S := \Spec{(k[[X]])}$ be proper, surjective and flat, such that the special fibre $\mathcal{C}_X$ is a semi-stable curve. Assume that for all $P \in \mathcal{C}_X$, which are singular in the special fibre, we have an isomorphism of $k[[X]]$-algebras $$\widehat{\O}_{\mathcal{C},P} \cong k[[X, a, b]] / (ab - X^n),$$ for some $n$ (depending on $P$), then $\mathcal{C}_{\eta} \rightarrow \eta$ is smooth.
\end{lemma}

\begin{proof}
By using \cite[Lem.\ 3.34, p.\ 102]{HarrisMorrison}, we find that $\mathcal{C}_{\eta}$ is a semi-stable curve. As the smooth locus is open, we know that $\mathcal{C}_{\eta}$ is smooth in the points specialising to smooth points of $\mathcal{C}_X$. To check that $\mathcal{C}_{\eta}$ is smooth in the points specialising to a singular point $P$ of $\mathcal{C}_X$, we recall the deformation theory of nodal singularities from \cite{DeligneMumford}. Since we have
$$\widehat{\O}_{\mathcal{C},P} \cong k[[X,a,b]] / (ab-X^n),$$
we see that $X^n = 0$ is the locus where the point is nodal, hence the points in the generic fibre $\mathcal{C}_{\eta}$ specialising to $P$ are smooth.
\end{proof}

\section{Proof of Theorem}

Now we can use the results from the previous section to construct smooth ordinary curves.

%\begin{theorem}\label{curvedeform}
%Let $G$ be a finite group and let $\gamma : C \rightarrow D$ be an orientation preserving Galois cover of semi-stable curves over $k$, with group $G$. Moreover, suppose that $C$ is ordinary. Then there exists a $G$-Galois cover $\Gamma : \mathcal{C} \rightarrow \mathcal{D}$ of semi-stable curves over $k[[X]]$ such that $\mathcal{C}_{\eta}$ and $\mathcal{D}_{\eta}$ are smooth, $\mathcal{C}_{\eta}$ is ordinary and $\Gamma_k \cong \gamma$, where $\eta$ is the generic point of $k[[X]]$.
%\end{theorem}
\begin{proof}(Theorem, p.\ \pageref{curvedeform})
Take a deformation $\mathcal{C} \rightarrow \mathcal{D}$ as in Proposition \ref{coverdeformation}. Then, if $C$ is ordinary, also $\mathcal{C}_{\eta}$ is ordinary, as the $p$-rank is lower semi-continuous, cf.\ \cite[sect.\ 2]{FaberGeer} and \cite[Th.\ 2.3.1]{Katz}.
\end{proof}

\begin{remark}\label{rmk:Fpbar}
Using an argument like Harbater's, see \cite[Thm.\ 2.7, p.\ 376]{IGT}, we can find an ordinary $G$-Galois cover over $k$ instead of the much larger field $k((X))$. The idea is that $\mathcal{C}_{\eta}$ is defined over a finitely generated algebra over $k$. By Bertini-Noether, the locus where the curves are irreducible is Zariski open and dense. Moreover, the ordinary and the smooth locus are also Zariski open and dense. Hence, as $k$ was assumed to be algebraically closed, we can find a point to specialise to in order to find an ordinary $G$-Galois cover of smooth curves over $k$.
\end{remark}

%% SECTION: EXAMPLES OF ORDINARY CURVES %%
\section{Examples of ordinary curves}

In this section, we will apply the theory developed in the previous chapter to construct examples of ordinary curves, which are Galois covers of some other curve. We will treat the following examples.

\begin{tabular}{c|c}
Galois group	&Galois cover of \\ \hline
$C_2$ (arbitrary genus hyperelliptic curve)	&projective line	\\
gen.\ by two elements, one of order 2, one of higher order	&elliptic curve	\\
$G \rtimes C_2$ for abelian groups $G$ as above	&projective line	\\
$D_n$	&projective line \\ 
$A_5$ &projective line \\
\end{tabular}

As a first example, we will show that there exist ordinary hyperelliptic curves of arbitrary genus in odd characteristic. The following reproves a result of Glass and Pries (\cite[Thm. 1, sect.\ 2, p.\ 301]{GlassPries}) using the tools that we developed.

\begin{example}\label{hyperellipticresult}
Let $p > 2$ be a prime number and let $g \geq 1$ an integer. Then there exists an ordinary hyperelliptic curve of genus $g$ over $k = \overline{\F}_p$.
\end{example}

\begin{proof}
We will prove this statement by induction on the genus. For $g=1$, in order to obtain an ordinary elliptic curve, take two copies of $\P^1$ and glue them in two points to obtain a curve $C$. Let $G := \Z/2\Z$ act on $C$ by swapping the two components. Then the quotient $D$ of $C$ by $G$ is isomorphic to $\P^1$ by Proposition \ref{quotientofcurves}. We deform this, using the main Theorem (p.\ \pageref{curvedeform}). The resulting curve $\mathcal{C}/\overline{\F}_p[[X]]$ is a $G$-cover of $\mathcal{D} = \P^1$. As the arithmetic genus is constant on flat families, the generic fibre of $\mathcal{C}$ gets the structure of an ordinary elliptic curve (by choosing any point as origin) over $\overline{\F}_p((X))$. Using Remark \ref{rmk:Fpbar}, we can use it to obtain an ordinary elliptic curve over $\overline{\F}_p$.

Now suppose that we managed to obtain an ordinary hyperelliptic curve $H/\overline{\F}_p$ of genus $\ell-1 \geq 1$. We will construct an ordinary hyperelliptic curve of genus $\ell$ out of this.

Take an ordinary elliptic curve $E$ over $\overline{\F}_p$. Let $\varphi_H : H \rightarrow \P^1$ and $\varphi_E : E \rightarrow \P^1$ be the associated 2:1-covers to $\P^1$. Pick points $P \in E(\overline{\F}_p)$ and $Q \in H(\overline{\F}_p)$ such that $\varphi_E$ (resp.\ $\varphi_H$) is ramified at $P$ (resp.\ $Q$). Let $C$ be the curve obtained by gluing $E$ and $H$ in $P$ and $Q$ respectively. Let $D$ be the curve obtained by gluing $\P^1$ and $\P^1$ in $\varphi_E(P)$ and $\varphi_H(Q)$, respectively.

Then there is a natural map $\varphi : C \rightarrow D$ and the action of $G := \Z/2\Z$ on $E$ and $H$ extends to $C$, giving $\varphi$ the structure of a $G$-Galois cover. Moreover, $C$ is ordinary as both $E$ and $H$ are, cf.\ Proposition \ref{ordinarycomponents}. Now the main Theorem (p.\ \pageref{curvedeform}) will give us a curve over $\overline{\F}_p[[X]]$, whose generic fibre will be a hyperelliptic curve of genus $\ell$ (as the arithmetic genus is constant in flat families) over $\overline{\F}_p((X))$. Using Remark \ref{rmk:Fpbar}, we then obtain an ordinary hyperelliptic curve of genus $\ell$ over $\overline{\F}_p$.
\end{proof}

\begin{example} \label{groupgenbytwoelems}
Let $G$ be a finite group, generated by two elements, of which one has order 2 and the other has order greater than 2 (e.g.\ $G = S_n$ or $G = D_n$ for $n \geq 3$). Let $p$ be a prime number coprime to $|G|$. Then there exists an elliptic curve $E$ over $\overline{\F}_p$ and a Galois cover $C \rightarrow E$ with group $G$ of smooth curves over $\overline{\F}_p$ such that $C$ is ordinary.
\end{example}

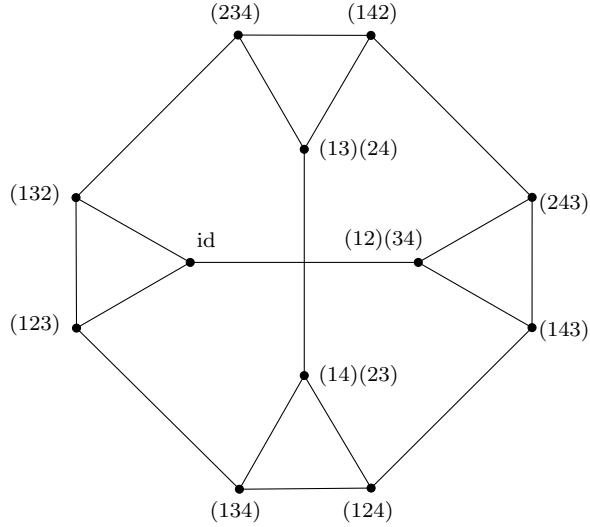
\begin{center}
\begin{tikzpicture}[line cap=round,line join=round,>=triangle 45,x=3.0cm,y=3.0cm]
\clip(0.2,1.2) rectangle (2.8,3.8);
\draw (1.5,2.)-- (1.215714836897041,1.4974723221713448);
\draw (1.215714836897041,1.4974723221713448)-- (1.7930591535529383,1.502537987919507);
\draw (1.7930591535529383,1.502537987919507)-- (1.5,2.);
\draw (0.5015367456649152,2.209661116220188)-- (1.,2.5);
\draw (1.,2.5)-- (0.4993275237727228,2.786512399217343);
\draw (0.4993275237727228,2.786512399217343)-- (0.5015367456649152,2.209661116220188);
\draw (1.5,3.)-- (1.792451388353112,3.5034444221143466);
\draw (1.792451388353112,3.5034444221143466)-- (1.2102300352319553,3.5049925427429973);
\draw (1.2102300352319553,3.5049925427429973)-- (1.5,3.);
\draw (2.,2.5)-- (2.4989232487301045,2.78762277784977);
\draw (2.4989232487301045,2.78762277784977)-- (2.4985502566900024,2.211731180985951);
\draw (2.4985502566900024,2.211731180985951)-- (2.,2.5);
\draw (1.7930591535529383,1.502537987919507)-- (2.4985502566900024,2.211731180985951);
\draw (2.4989232487301045,2.78762277784977)-- (1.792451388353112,3.5034444221143466);
\draw (1.2102300352319553,3.5049925427429973)-- (0.4993275237727228,2.786512399217343);
\draw (0.5015367456649152,2.209661116220188)-- (1.215714836897041,1.4974723221713448);
\draw (1.5,2.)-- (1.5,3.);
\draw (1.,2.5)-- (2.,2.5);
\begin{scriptsize}
\draw [fill=black] (1.,2.5) circle (1.5pt);
\draw [fill=black] (0.5015367456649152,2.209661116220188) circle (1.5pt);
\draw [fill=black] (1.5,3.) circle (1.5pt);
\draw [fill=black] (1.2102300352319553,3.5049925427429973) circle (1.5pt);
\draw [fill=black] (1.5,2.) circle (1.5pt);
\draw [fill=black] (1.7930591535529383,1.502537987919507) circle (1.5pt);
\draw [fill=black] (2.4985502566900024,2.211731180985951) circle (1.5pt);
\draw [fill=black] (2.,2.5) circle (1.5pt);
\draw [fill=black] (0.4993275237727228,2.786512399217343) circle (1.5pt);
\draw [fill=black] (1.792451388353112,3.5034444221143466) circle (1.5pt);
\draw [fill=black] (1.215714836897041,1.4974723221713448) circle (1.5pt);
\draw [fill=black] (2.4989232487301045,2.78762277784977) circle (1.5pt);
\draw[color=black] (1.07,2.6) node {id};
\draw[color=black] (0.32,2.225) node {(123)};
\draw[color=black] (1.74,3.) node {(13)(24)};
\draw[color=black] (1.2,3.6) node {(234)};
\draw[color=black] (1.74,2.) node {(14)(23)};
\draw[color=black] (1.78,1.4) node {(124)};
\draw[color=black] (2.64,2.2) node {(143)};
\draw[color=black] (1.85,2.6) node {(12)(34)};
\draw[color=black] (0.32,2.8) node {(132)};
\draw[color=black] (1.8,3.6) node {(142)};
\draw[color=black] (1.2,1.4) node {(134)};
\draw[color=black] (2.64,2.76) node {(243)};
\end{scriptsize}
\end{tikzpicture}
\\
{\em An illustration of the graph constructed in the proof of Example \ref{groupgenbytwoelems} with $G = A_4$ and generators $h_1 = (12)(34)$ and $h_2 = (123)$.}
\end{center}

\begin{proof}
Let $h_1, h_2$ be two generators of $G$, of which $h_1$ is of order 2 and $h_2$ of higher order. Then we can consider the graph $\Gamma$, whose vertices are elements $g \in G$ and for each vertex $g \in G$, the edge end set $E_g$ consists of four elements, one of which being $\emptyset_g$, such that the other three edge ends are connected to edge ends of $gh_1$, $gh_2$ and $gh_2^{-1}$, respectively. This graph has the property that every vertex is connected to exactly three other vertices. The group $G$ acts on the graph by 
$$G \rightarrow \mathrm{Aut}(\Gamma) : g \mapsto (h \mapsto gh).$$

Next, we will construct a stable curve $C$ out of this graph. As components we take copies of $\P^1$'s, one for each element $g \in G$, glued together in any arbitrary way such that $\Gamma$ is the graph associated to $C$. For any pair of triples of distinct points in $\P^1$, there is a unique automorphism sending the first triple of points to the second triple. In this way, we can extend the action of $G$ on $\Gamma$ to an action of $G$ on $C$. Using Proposition \ref{quotientofcurves} we find that the quotient $C/G$ is isomorphic to $\P^1$ glued to itself in one point. Now we can use Proposition \ref{coverdeformation} to deform this cover into a $G$-Galois cover of an elliptic curve $E$ by an ordinary smooth curve over $\overline{\F}_p((X))$. Using Remark \ref{rmk:Fpbar}, we obtain the desired cover of smooth curves over $\overline{\F}_p$.
\end{proof}

\begin{remark}
Any non-abelian finite simple group can be realised in this way. Because of the odd order theorem, such a group $G$ has even order. Take any element $h_1$ of order 2. Then there exists, cf.\ \cite{Stein}, another element $h_2$ such that $h_1$ and $h_2$ generate $G$. By taking $h_1 \cdot h_2$ instead of $h_2$ if necessary, we may assume that $h_2$ has order greater than 2.
\end{remark}

\begin{example}\label{abgroupgenbytwoelems}
If in addition to the hypotheses in Example \ref{groupgenbytwoelems} the group $G$ is abelian, let $H$ be the group $G \rtimes C_2$, where $C_2$ acts on $G$ by inverting all elements. Then there exists a Galois cover $C \rightarrow \P^1$ with group $H$, such that $C$ is a smooth ordinary curve over $\overline{\F}_p$.
\end{example}

\begin{proof}
In the case $G$ is abelian, also $$\Gamma \rightarrow \Gamma : g \mapsto g^{-1}$$ is an automorphism of $\Gamma$. Using this automorphism, we obtain an action of $H$ of $\Gamma$. Again the action extends to $C$ and the quotient $C/H$ is isomorphic to $\P^1$. Using Proposition \ref{coverdeformation} and Remark \ref{rmk:Fpbar} again, we get a $H$-Galois cover of $\P^1$ by an ordinary smooth curve over $\overline{\F}_p$.
\end{proof}

\begin{example}\label{example:ngon}
Let $n$ be an integer and let $p$ be a prime number coprime to $2n$. Then there exists a Galois cover $C \rightarrow \P^1$ with group $D_n$, such that $C$ is a smooth ordinary curve over $\overline{\F}_p$.
\end{example}

\begin{proof}
Consider the regular $n$-gon as a graph $\Gamma$ and let $D_n$ act on it. By gluing $n$ copies of $\P^1$ subsequently in the points $0$ and $\infty$ (gluing the $0$ of one curve to the $\infty$ of the next one), we can construct a semi-stable curve $C$ whose associated graph is $\Gamma$.

\begin{center}
\begin{tikzpicture}[line cap=round,line join=round,>=triangle 45,x=1.0cm,y=1.0cm]
\clip(-1.,-1.) rectangle (5.,5.);
\draw [shift={(-4.,2.)}] plot[domain=-0.7853981633974483:0.7853981633974484,variable=\t]({1.*4.47213595499958*cos(\t r)+0.*4.47213595499958*sin(\t r)},{0.*4.47213595499958*cos(\t r)+1.*4.47213595499958*sin(\t r)});
\draw [shift={(2.,8.)}] plot[domain=5.497787143782138:5.497787143782138,variable=\t]({1.*4.47213595499958*cos(\t r)+0.*4.47213595499958*sin(\t r)},{0.*4.47213595499958*cos(\t r)+1.*4.47213595499958*sin(\t r)});
\draw [shift={(2.,8.)}] plot[domain=3.9269908169872414:5.497787143782138,variable=\t]({1.*4.47213595499958*cos(\t r)+0.*4.47213595499958*sin(\t r)},{0.*4.47213595499958*cos(\t r)+1.*4.47213595499958*sin(\t r)});
\draw [shift={(8.,2.)}] plot[domain=2.356194490192345:3.9269908169872414,variable=\t]({1.*4.4721359549995805*cos(\t r)+0.*4.4721359549995805*sin(\t r)},{0.*4.4721359549995805*cos(\t r)+1.*4.4721359549995805*sin(\t r)});
\draw [shift={(2.,-4.)}] plot[domain=0.7853981633974483:2.356194490192345,variable=\t]({1.*4.47213595499958*cos(\t r)+0.*4.47213595499958*sin(\t r)},{0.*4.47213595499958*cos(\t r)+1.*4.47213595499958*sin(\t r)});
\begin{scriptsize}
\draw [fill=black] (0.,4.) circle (1.5pt);
\draw [fill=black] (0.,0.) circle (1.5pt);
\draw [fill=black] (4.,4.) circle (1.5pt);
\draw [fill=black] (4.,0.) circle (1.5pt);
\end{scriptsize}
\end{tikzpicture}
\\
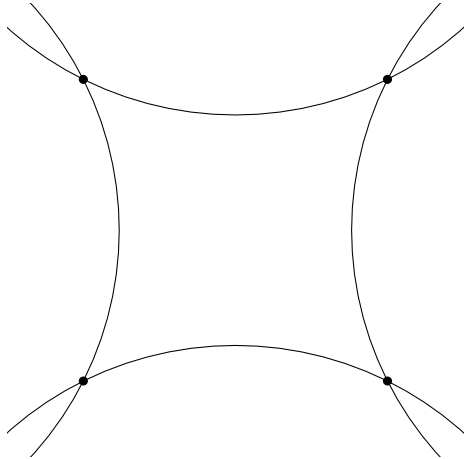
{\em An example with $n=4$: four copies of $\P^1$ glued in a square form.}
\end{center}

The group $D_n$ acts on $\Gamma$ in a natural way and this action can be extended to $C$, using the automorphism $(x:y) \mapsto (y:x)$ to mirror the sides of the $n$-gon, if necessary. The quotient $C/D_n$ is isomorphic to $\P^1$ (cf.\ Proposition \ref{quotientofcurves}) and by deforming it, using Proposition \ref{coverdeformation} and Remark \ref{rmk:Fpbar}, we find a $D_n$-Galois cover of $\P^1$ by an ordinary smooth curve over $\overline{\F}_p$.
\end{proof}

\begin{example}\label{example:A5}
Let $p > 5$ be a prime number. Then there exists a Galois cover $C \rightarrow \P^1$ with group $A_5$ such that $C$ is a smooth ordinary curve over $\overline{\F}_p$.
\end{example}

\begin{center}
{\begin{tikzpicture}[line cap=round,line join=round,>=triangle 45,x=1.5cm,y=1.5cm]
\clip(-1.4,-1.) rectangle (2.4,2.7);
\draw (-0.5877852522924734,-0.8090169943749473)-- (-1.2600735106701006,1.2600735106701015);
\draw (-1.2600735106701006,1.2600735106701015)-- (0.5,2.538841768587626);
\draw (0.5,2.538841768587626)-- (2.260073510670101,1.2600735106700998);
\draw (2.260073510670101,1.2600735106700998)-- (1.587785252292473,-0.8090169943749477);
\draw (1.587785252292473,-0.8090169943749477)-- (-0.5877852522924734,-0.8090169943749473);
\draw (-0.5877852522924734,-0.8090169943749473)-- (0.,0.);
\draw (0.,0.)-- (0.5,1.5388417685876266);
\draw (0.5,1.5388417685876266)-- (1.,0.);
\draw (1.,0.)-- (-0.30901699437494734,0.9510565162951536);
\draw (-0.30901699437494734,0.9510565162951536)-- (1.3090169943749475,0.9510565162951532);
\draw (1.3090169943749475,0.9510565162951532)-- (0.,0.);
\draw (-0.30901699437494734,0.9510565162951536)-- (-1.2600735106701006,1.2600735106701015);
\draw (0.5,1.5388417685876266)-- (0.5,2.538841768587626);
\draw (1.3090169943749475,0.9510565162951532)-- (2.260073510670101,1.2600735106700998);
\draw (1.,0.)-- (1.587785252292473,-0.8090169943749477);
\begin{scriptsize}
\draw [fill=black] (0.,0.) circle (1.5pt);
\draw [fill=black] (1.,0.) circle (1.5pt);
\draw [fill=black] (1.3090169943749475,0.9510565162951532) circle (1.5pt);
\draw [fill=black] (0.5,1.5388417685876266) circle (1.5pt);
\draw [fill=black] (-0.30901699437494734,0.9510565162951536) circle (1.5pt);
\draw [fill=black] (-0.5877852522924734,-0.8090169943749473) circle (1.5pt);
\draw [fill=black] (1.587785252292473,-0.8090169943749477) circle (1.5pt);
\draw [fill=black] (2.260073510670101,1.2600735106700998) circle (1.5pt);
\draw [fill=black] (0.5,2.538841768587626) circle (1.5pt);
\draw [fill=black] (-1.2600735106701006,1.2600735106701015) circle (1.5pt);
\end{scriptsize}
\end{tikzpicture}
\\
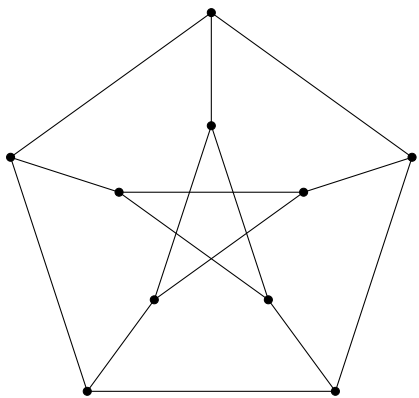
{\em The Petersen graph.}}
\end{center}

\begin{proof}
This time take $\Gamma$ to be the Petersen graph (see the figure). Its automorphism group is $S_5$. We construct the curve $C$ by taking 10 copies of $\P^1$ and glue them arbitrarily to obtain a stable curve, whose associated graph is $\Gamma$. We extend the action of the subgroup $A_5 \subset S_5$ on $\Gamma$ to $C$, like in the previous examples. This action satisfies all the necessary conditions. The quotient $C/A_5$ is $\P^1$ and by deforming again, we find a $A_5$-Galois cover of $\P^1$ by an ordinary smooth curve.
\end{proof}

%\begin{remark}
%Due to Remark \ref{rmk:Fpbar}, the field $K$ referred to in Example \ref{hyperellipticresult}, \ref{groupgenbytwoelems}, \ref{abgroupgenbytwoelems}, \ref{example:ngon} and \ref{example:A5} can be taken to be $\overline{\F_p}$.
%\end{remark}

%% SECTION : BIBLIOGRAPHY %%

\end{document}